\documentclass[12pt]{amsart} 

\usepackage{amssymb}

\usepackage{thmtools, thm-restate}

\newtheorem{thmA}{Theorem}

\newtheorem{propA}{Proposition}

\newtheorem{theorem}{Theorem}[section]
\newtheorem{lemma}[theorem]{Lemma}

\newtheorem{cor}[theorem]{Corollary}
\newtheorem{corollary}[theorem]{Corollary}
\newtheorem{prop}[theorem]{Proposition}
\newtheorem{proposition}[theorem]{Proposition}

\theoremstyle{definition}
\newtheorem{definition}[theorem]{Definition}

\theoremstyle{remark}
\newtheorem{remark}[theorem]{Remark}

\numberwithin{equation}{section}

\def\cat{{\rm{CAT}}$(0)$ }

\def\n{\hbox{\bf{n}}}
\def\C{\mathcal C}
\def\NC{\mathcal N(\mathcal C)}
\def\NCp{\mathcal N(\mathcal C')}
\def\NT{\mathcal N(\mathcal T)}
\def\T{\mathcal T}
\def\ssm{\smallsetminus}
\def\e{\varepsilon}

\def\im{\hbox{im }}
\def\R{\mathbb R}

\def\E{\mathbb E}
\def\g{\gamma}
\def\s{\sigma}
\def\e{\varepsilon}
\def\sym{{\rm{sym}}}
\def\out{{\rm{Out}}(F_n)}

\def\fix{\hbox{\rm{Fix}}}

\def\l{\lambda}

\def\r{\rho}

\def\<{\langle}
\def\>{\rangle}

\def\-{\overline}

\def\N{\mathbb N}
\def\Ne{\mathcal N}
\def\Ni{{\rm{Niel}}}
\def\Z{\mathbb Z}

\def\S{\mathbb S}
 
\def\G{\Gamma}
\def\fd{\hbox{\rm{FixDim}}}
\def\sln{\hbox{\rm{SL}}(n,\mathbb Z)}
\def\slnz{$\hbox{\rm{SL}}(n,\mathbb Z)$}
\def\autn{$\hbox{\rm{Aut}}(F_n)$ }
\def\aut{\text{\rm{Aut}} (F_n)} 

\def\outn{\text{\rm{Out}} (F_n)} 
\def\FR{\text{\rm{F}}\mathbb R}
\def\isom{\text{\rm{Isom}}}

\def\sautn{\text{\rm{SAut}}(F_n)}
\def\saut{\text{\rm{SAut}}(F_n)}
\def\modg{\text{\rm{Mod}}_g}

\def\A{\mathcal A}

\catcode`\@=11
\def\serieslogo@{\relax}
\def\@setcopyright{\relax}
\catcode`\@=12

\begin{document}

\title[Helly-type theorems, CAT$(0)$ spaces, and actions of 
Aut$(F_n)$]
{Helly-type theorems, CAT$(0)$ spaces, and actions of 
automorphism groups of free groups} 
\author[Martin R. Bridson ]{Martin R.~Bridson}
\address{Mathematical Institute, AWB Woodstock Road, Oxford OX2 6GG, Europe}
\email{bridson@maths.ox.ac.uk} 




\thanks{ 
For the purpose of open access, the author has applied a CC BY public copyright licence to any author accepted manuscript arising from this submission.} 

\begin{abstract}  
We prove a variety of fixed-point theorems for
groups acting on CAT$(0)$ spaces. Fixed points are obtained by a bootstrapping
technique, whereby increasingly large subgroups are proved to have fixed points:
specific configurations in the subgroup lattice of $\G$ are exhibited and
Helly-type theorems are developed to prove that the fixed-point sets of the subgroups in 
the configuration intersect. In this way, we obtain
lower bounds on the smallest dimension $\fd(\G)+1$ in which various
groups of geometric interest  can act on a complete CAT$(0)$ space without a global
fixed point.
For automorphism groups of free groups, we prove $\fd(\aut)\ge \lfloor 2n/3\rfloor$.  
\end{abstract}

\maketitle
  
In this article we shall prove fixed-point theorems for groups
acting on \cat spaces by analyzing the pattern of fixed-point sets of
subgroups.  
 The basic question that we address is this:
given a group $\G$, what is the least integer $d=\fd(\G)+1$ such
that $\G$ admits a fixed-point-free
 action by isometries on a complete \cat space of dimension $d$? 
 
 We shall present a number
 of general results and methods for establishing bounds on $d$ and then apply them to groups of
 geometric interest. 
 We shall pay particular attention to the automorphism groups  of free
 groups, ${\rm{Aut}}(F_n)$,  but many other groups of geometric interest
 will  enter the discussion, such as higher-rank lattices, mapping class groups,
 and braid groups. Besides their intrinsic interest, these examples are pursued in
 detail in order to illustrate the practical nature of the general methods
 that we develop, particularly
 the {\em{Ample Duplication Criterion}}, whose technical statement we defer for the moment. 
 
The proof of the following theorem is the most involved in
this paper. It requires a subtle and iterated use of the
various fixed-point criteria that we develop, as well as a detailed understanding of generating sets for $\aut$.

\begin{restatable}{thmA}{firstthmA}
\label{i:main}%
If $n\ge 3m$ and $d<2m$, or $n\ge 3m+2$ and $d<2m+1$,
then $\aut$ has a fixed point whenever it acts by isometries on a complete 
${\rm{CAT}}(0)$ space of dimension $d$. 
\end{restatable}

We obtain the same bound for  $\saut$, the unique subgroup of index $2$ in $\aut$.
An important point to note is that
we do not assume that the \cat spaces we study are locally compact, nor do we assume that
the actions are by semisimple isometries.
As extra conditions are imposed on the space and the action, sharper results 
are obtained. For example, we shall prove
that $\saut$ cannot act non-trivially by semisimple isometries on any smooth, complete
\cat manifold of dimension less than $2n-4$ (Theorem \ref{t:hadamard}). 
For ${\rm{SL}}(n,\Z)$, if $n\ge 3$ then the group has a fixed point
whenever it acts by semisimple isometries on a complete \cat space of finite dimension (Proposition \ref{p:sln})
but for actions that are not semisimple we only know that $\fd(\sln)$ is at least $n-2$ when $n$
is odd and at least $n-3$ when $n$ is even (Proposition \ref{t:sln}).
 
\subsection*{General Criteria}

If a group $H$ acts by isometries on a complete \cat space $X$, then  
the points of $X$ fixed by $H$ form a closed convex subspace. The dimension\footnote{topological
covering dimension} of a \cat space 
places constraints on
the way in which families of closed convex subsets of  \cat spaces can intersect. The most classical
instance of this is Helly's Theorem \cite{helly}: if one has a finite collection of  
convex subsets in
$\R^d$, and each $(d+1)$-member sub-collection has a non-empty intersection, then the entire family has
non-empty intersection. There are many proofs and many generalisations of this theorem in the literature,
often couched in homological language, as in the Acyclic Covering Lemma
(see \cite{brown}  p.168, for example).  
For our purposes, the most useful generalisation is the following, 
which is a special case of the version whose proof is given in an appendix to this paper.
Recall that the {\em nerve} $\NC$ of a 
collection of subsets $\C$ is the simplicial complex whose $k$-simplices $[i_0,\dots,i_k]$ correspond to sub-collections
$\{C_{i_0},\dots, C_{i_k}\}\subset\C$ with non-empty intersection.

\begin{thmA} \label{t:helly} Let $X$ be a complete  convex metric space (for example
a complete \cat space) and let $\C$ be a finite collection of closed convex subsets of $X$.
If  $\dim(X)\le d$, then every continuous map $\NC\to\S^r$ from the nerve of $\C$ to a
sphere of dimension $r\ge d$ is homotopic to a constant map.
\end{thmA}
  
One recovers the classical Helly Theorem by taking $X=\R^d$,  noting that if Helly's  Theorem failed then
a counterexample $\C$ of minimal cardinality $s\ge d+2$ would have $\NC = \partial\Delta_{s-1}\approx\S^{s-2}$ --
cf.~Corollary \ref{empty}. 
  
We apply Theorem \ref{t:helly} to the fixed-point sets of groups of isometries. 
Our strategy
is to prove fixed-point theorems for groups of geometric interest by induction, analyzing configurations of 
subgroups that can be more-easily proved to have fixed points.  A simple illustration of this is the following (Proposition \ref{p:Delta}). {\em{Let $\G$ be a group that is
 generated by  $A_1\cup\dots\cup A_m\subset\G$
 and let $X$ be a complete \cat space of dimension $d$ on which $\G$
 acts by isometries. If the subgroup generated by each
 $(d+1)$ of the sets $A_i$ has a fixed point in $X$, then
 $\G$ has a fixed point.}}

When applying such results, a useful starting point is the observation that finite groups of isometries of
complete  \cat spaces always have fixed points. This underpins a number of 
simply stated fixed-point results, such as:

\begin{propA}[Product Lemma with Torsion]\label{p:prodLemma}  If the
groups $\G_1,\dots,\G_d$ each have a finite generating set consisting of elements of 
finite order, then at least one of the $\G_i$ has a fixed point whenever 
$\G_1\times\dots\times\G_d$ acts by isometries on a complete \cat space of dimension
less than $d$. 
\end{propA}

From such elementary observations, one quickly obtains results such as the following  
(Proposition \ref{p:bieb} and \ref{p:hyp}).

\begin{propA}\label{t:bieb} There exist groups $(\G_n)_{n\in\N}$
such that $\G_n$ acts properly and cocompactly by
isometries on $\E^n$ but cannot act without a fixed
point on any complete \cat space of dimension less than $n$.
There also exist hyperbolic groups $(\Lambda_n)_{n\in\N}$ with $\fd(\Lambda_n)=n$.
\end{propA}
 
\smallskip
The {\em{bootstrapping}} technique introduced in Section \ref{s:bootstrap} leads
to more subtle and powerful fixed-point criteria.
\begin{propA}[Bootstrap Lemma]\label{p:bootstrap}
Let $k_1,\dots,k_n$ be positive
integers and let $X$ be a complete \cat space of dimension less than $k_1+\dots+k_n$.
Let $S_1,\dots,S_n \subset\isom(X)$ be subsets with $[s_i,s_j]=1$
for all $s_i\in S_i$ and $s_j\in S_j\ (i\neq j)$.

If, for $i=1,\dots,n$, each $k_i$-element subset of $S_i$ has a 
fixed point in $X$, then for some $i$ every 
finite subset of $S_i$ has a fixed point.
\end{propA}  
The most powerful tool that we develop is  
the following {\em{Ample Duplication Criterion}};
 it is used extensively in this article and applied
to mapping class groups in \cite{mb:dim}. I expect that it will have many further applications.
 
Given a group $\G$ and positive
integers $d$ and $k_0$, we say that a finite generating set
$\A$ for a subgroup $\Lambda<\G$ has {\em{ample duplication
for dimension $d$, with base $k_0$,}} if there is a function  
$f:\mathbb{N}\to \mathbb{N}$ such that
the following conditions hold:
\begin{enumerate}
\item  Each subset $S\subset\A$ of cardinality $|S|> k_0$ can either be written
as a disjoint union $S=S_1\sqcup S_2$ where the $S_i$ are non-empty
and $\langle S_1\rangle$ normalizes $\langle S_2\rangle$,
or else there are at least $f(|S|)$ commuting conjugates of 
$\langle S\rangle$ in $\G$;
\item $d < (k-1)\,f(k)$ for $k=k_0+1,\dots,\min\{d+1,\, |\A|\}$. 
\end{enumerate}

\begin{thmA}[Ample Duplication Criterion]\label{p:scheme} 
Let $\G$ be a group acting by isometries on a complete \cat space $X$ of dimension
at most $d$, and let $\Lambda<\G$ be a subgroup with
a finite generating set $\A$ that has ample duplication
for dimension $d$, with base $k_0$. If  $\<S\>$ has a fixed point for every $S\subset\A$ with $|S|\le k_0$, then $\Lambda$ has a fixed point in $X$.
\end{thmA}

\begin{remark} All of the fixed-point results displayed above   
actually hold in greater generality.  They
are valid for actions  on finite dimensional,  contractible metric spaces $X$ with the following properties:
(1) if $\G<{\rm{Isom}}(X)$ has a bounded orbit,  then it has a fixed point; and  
(2) the intersection of 
the fixed-point sets for any finite collection of subgroups $H_1,\dots, H_n<{\rm{Isom}}(X)$ is contractible if it
is non-empty (in fact one needs something less than contractibility -- see Theorem \ref{nerve}).
I have chosen to state the fixed-point results in the \cat setting because I believe
that it makes them more immediately engaging and because this is where the main interest lies.
Nevertheless,   the proofs  
are constructed so as to make it clear that conditions (1) and (2) suffice.  The only results in this article that require
additional properties of CAT$(0)$ spaces are Theorem \ref{t:hadamard} and Corollaries \ref{c:mcg} and \ref{c:ss}.
\end{remark}

\subsection*{History and Comparison}
In the remainder of this introduction I shall explain how the results established here relate to earlier work 
of a similar nature.
Following Serre \cite{serre},
one says that a group $\Gamma$ has {\em{property FA}}
if it cannot act without a  fixed point on any simplicial tree, and
following  \cite{CV} one says that $\Gamma$  
has {\em{property $\FR$}} if it fixes a point whenever it acts by isometries on an $\mathbb{R}$-tree 
 (i.e.~a complete CAT$(0)$ space of topological
dimension 1). In our terminology, $\G$ has property
$\FR$ if and only if $\fd(\G)\ge 1$. There are finitely generated groups that have
FA but not $\FR$; see \cite{minasyan}. Serre \cite{serre} proved
that $\sln$ has property $\FR$ if $n\ge 3$, and Bogopolski \cite{oleg}
and Culler-Vogtmann \cite{CV} strengthened this by proving that $\aut$ and $\outn$ have $\FR$
if $n\ge 3$.  These proofs rely in an essential way on
the fact that isometries of $\R$-trees are semisimple, but there is a later proof which does not rely on this fact \cite{mb:zies},
and that is the forerunner of what we do here.  For $n\ge 5$, 
property $({\rm{T}})$,  established by Kaluba, Kielak and Nowak \cite{KKN},  tells us that $\out$
cannot act without a fixed point on any finite dimensional CAT$(0)$ cube complex \cite{KKN}. 

 In \cite{farb} Farb  considered a generalization of
property FA. He defines a group $\G$ to have property ${\rm{FA}}_n$ if
it fixes  a point whenever it acts by simplicial isometries on a 
CAT$(0)$ piecewise-Euclidean complex of dimension at most $n$
that has only finitely many isometry types of cells.
By exploiting a homological version of Helly's theorem, we was
able to prove that various groups of geometric interest have
property ${\rm{FA}}_n$.  He also considered more general actions on
non-polyhedral spaces, but retained the condition that the action must be by semisimple isometries.
(Cellular actions on polyhedral complexes with only finitely many isometry types of cells
are necessarily by semisimple isometries \cite{mb:pams}.)

We dispense with these conditions and
consider instead actions by isometries
on arbitrary complete, finite-dimensional CAT$(0)$ spaces, 
placing no other 
constraints on the structure of the space or on the type
of the action.  This greater generality is important because  many
of the groups that we wish to consider, such as $\sln, n\ge 3$,   
admit interesting actions with parabolics on finite dimensional CAT$(0)$ spaces (e.g.~the symmetric 
space for ${\rm{SL}}(n,\R)$) but have a fixed point whenever they act by semisimple isometries (see Proposition \ref{p:sln}).

I realised one could upgrade the ideas from \cite{mb:zies} to prove fixed-point
theorems in higher dimensions after hearing Benson
Farb's lecture on ${\rm{FA}}_n$ in Neuchatel in the summer of 2000, and I first
presented Theorem \ref{i:main} at the Oberwolfach meeting
on Geometric Methods in Group Theory later that year.
My intermittent efforts to improve the bounds
in the intervening years have yielded a number of related results, 
some of which appeared in \cite{mb:bill}, \cite{mb:dim}, \cite{gang6}, \cite{mb:profinite}
 and \cite{mb:rddh}, but 
I have been unable to improve the bound in Theorem \ref{i:main} and I apologise for waiting so
long to publish the proof. 
In the meantime, Olga Varghese \cite{olga} wrote a simpler version of the argument for $\aut$ 
(avoiding the Ample Duplication Criterion) 
that yields a weaker bound, and the use of Helly-type theorems to prove fixed-point results has found
many further applications, for example \cite{Ye1, Ye2}.

\noindent{\bf{Acknowledgements:}} 
This research was supported by a Fellowship from the EPSRC and by a Royal
Society Wolfson Research Merit Award. During the preparation of this manuscript, I benefited from the hospitality of the EPFL (Lausanne), l'Universit\'e de Gen\`eve, the Mittag Leffler Institute,  the Hausforff Institute in Bonn, and Stanford University.
I thank them all.  I thank Anders Karlsson,  Dawid Kielak and Andrew Putman for helpful correspondence,  and I
thank the referees for their careful reading and thoughtful comments.

\section{Isometries of \cat spaces}

In this section I'll gather the basic facts that we'll need about isometries of 
\cat spaces. 
The standard reference for this material is \cite{BH}.

Let $X$ be a  geodesic metric space.
A geodesic triangle $\Delta$ in $X$ consists of three points $a,b,c\in X$ and 
three geodesics $[a, b],\, [b,c],\, [c,a]$. Let $\-\Delta\subset\E^2$ be a triangle
in the Euclidean plane with the same edge lengths as $\Delta$
and let $\overline x\mapsto
x$ denote the map  $\-\Delta\to\Delta$
that sends each side of $\-\Delta$ isometrically
onto the corresponding side of $\Delta$. 
One says that {\em{$X$ is a 
{\rm{CAT}}$(0)$ space}} if for all $\Delta$ and all $\-x,\-y\in\-\Delta$ the inequality
 $d_X(x,y)\le d_{\E^2}(\-x, \-y)$ holds.

In a {\rm{CAT}}$(0)$ space there is a unique geodesic $[x,y]$ joining
each pair of points $x,y\in X$.
A subspace $Y\subset X$ is said to be {\em{convex}} if 
$[y,y']\subset Y$ whenever $y,y'\in Y$.  

Given a subset $H\subset\isom(X)$, we denote its
set of common  fixed points by
$$\fix(H):=\{x\in X\mid \forall h\in H,\  h.x=x\}.$$
Note that $\fix(H)$  is closed and convex.

The isometries of a {\rm{CAT}}$(0)$ space $X$ divide naturally into
 {\em{semisimple}} isometries, i.e. those for which there exists $x_0\in X$ such that
$d(\gamma.x_0, \, x_0) = |\g|$ where $|\g|:=\inf\{d(\gamma.y, \, y) \mid y\in X\}$, and the remainder, which
are said to be
{\em{parabolic}}.  
 Semisimple isometries are divided into {\em{hyperbolics}}, for which $|\g|>0$,
and {\em{elliptics}}, which have fixed points.

Any bounded set in a complete \cat space $X$ is contained in
a unique closed ball of smallest radius (\cite{BH}, p.178). If the bounded set
is the orbit of a point under the action of a group $\G$
acting by isometries,
then the centre of the ball will be fixed by $\G$. Thus we have:

\begin{prop}\label{bddOrbit}
 If $X$ is a complete \cat space and $\G<\isom(X)$, then the following conditions are equivalent:
 \begin{itemize}
 \item $\G$ has a bounded orbit in $X$;
 \item all $\G$-orbits in $X$ are bounded;
 \item   $\fix(\G)$ is non-empty.
 \end{itemize}
\end{prop}

\begin{cor} \label{fp}\label{c:finite}
Whenever a finite group  acts by isometries
on a complete  \cat space, it fixes a point.
\end{cor}

\subsection*{Normalised and commuting subgroups}

\begin{prop}\label{p:norm} Let $\G$ be a group,
let $H_1,H_2<\G$ be subgroups, and suppose
that $H_2$ normalizes $H_1$.
Whenever $\G$ acts by
isometries on a complete \cat space, if $\fix(H_1)$
and $\fix(H_2)$ are non-empty, then
$\fix(H_1)\cap \fix(H_2)$ is non-empty.
\end{prop}

\begin{proof}  Because $H_2$ normalizes $H_1$, the orbit of any $x\in \fix(H_1)$ under
$\<H_1, H_2\>$ is simply the $H_2$-orbit of $x$.  Since $\fix(H_2)$ is non-empty,
this orbit is bounded.   Thus we may appeal to Proposition \ref{bddOrbit}.
\end{proof}

\begin{cor}\label{c:commutes}
 Let $X$ be a complete  {\rm{CAT}}$(0)$ space. If
the subgroups $H_1,\dots, H_\ell$ of $\isom(X)$ pairwise commute and
$\fix(H_i)$ is non-empty for $i=1,\dots,\ell$, then 
$\bigcap_{i=1}^\ell \fix(H_i)$ is non-empty.
\end{cor}

\subsection*{Subgroups of finite index}

\begin{lemma}\label{l:fi} Let $H<\G$ be a subgroup of finite
index. If $\G$  acts by isometries on a complete \cat space and
$\fix(H)$ is non-empty, then $\fix(\G)$ is non-empty.
\end{lemma}

\begin{proof} 
If $x$ is a fixed point of $H$, then the $\G$-orbit of $x$ is finite, so Proposition \ref{bddOrbit} applies.
\end{proof}

\begin{corollary}\label{fixFI}
If $H<\G$ has finite index, then $\fd(H)\le\fd(\G)$.
\end{corollary}
  
This inequality can be strict. For example,  it can be deduced from Corollary \ref{c:propA+} that
for $\G=\text{SL}(2,\Z)\wr C_n$ one has $\fd(\G)=n-1$, if $C_n$ is the
cyclic group of order $n$. But $\G$
contains $H=\text{SL}(2,\Z)\times
\dots\times \text{SL}(2,\Z)$ as a subgroup of finite index, and
this acts non-trivially on a tree via its projection to any factor,
so $\fd(H)=0$.

\section{Dimension, Nerves, and Helly-type Theorems}\label{s:helly}

In this section we record the basic facts about dimension 
that we need, along with the
consequences of Theorem \ref{t:helly} that will be most useful for us.

The treatise of Hurewicz and Wallmann \cite{HW} summarizes the
classical results on dimension theory due to Borsuk, Kuratowski and others,
proving in particular that various definitions of dimension (inductive covering
dimension, Cech cohomological dimension, etc.) are equivalent for second countable
metric spaces. For our purposes, the most useful definition is the following.

\begin{definition} A topological space $X$ has dimension
at most $d$, written $\dim(X)\le d$, if for every closed
subspace $K\subseteq X$, every continuous map
$f:K\to\S^r$ to a sphere of dimension $r\ge d$ extends to
a continuous map $X\to\S^r$.
\end{definition}

\subsection{Spaces with convex metrics}

The metric
on a geodesic space $X$ is  {\em{convex}} if 
for any pair of geodesics $c_1,c_2:[0,1]\to X$ parameterized
proportional to arc length, $t\mapsto d(c_1(t),\, c_2(t))$ is
a convex function.  In this circumstance, we say that $X$ is a {\em{convex metric space}}. 
These spaces were studied by Busemann \cite{bus}.
\cat spaces are convex in this sense \cite{BH}, p.120. 
In a convex space there is a unique geodesic
segment joining each pair of points, and the obvious retraction along
geodesics to an arbitrary basepoint shows that the space is
contractible.

The argument by which one deduces the classical Helly Theorem from Theorem \ref{t:helly}
can be abstracted as follows.

\begin{definition}\label{d:empty}
Given a simplicial
complex $K$ and a set of vertices $V\subset K^{(0)}$,
we say that $V$ {\em{spans an empty $r$-simplex}}
if $|V|=r+1$ and every proper subset of $V$ is the vertex
set of a simplex in $K$ but $V$ itself is not.
\end{definition}

\begin{corollary}[No Empty Simplices] \label{empty}
If $X$ is a complete convex metric space of 
topological dimension $\dim (X)\le d$ and  $\C$  
is a  finite collection of closed convex subsets, then the nerve
$\NC$ has no empty $r$-simplices for $r>d$.
\end{corollary}

\begin{proof} An empty $r$-simplex would correspond to a sub-collection $\C'\subseteq\C$ with $\NCp = \partial\Delta_r\approx \S^{r-1}$. Applying Theorem \ref{t:helly} to $\C'$, we obtain a contradiction.
\end{proof}

The following consequence is well known.
 
\begin{corollary}[Metric Helly]\label{metric:helly}  Let $C_1,\dots,C_m$ be closed convex
subspaces in a convex metric space $X$. If $\dim(X)\le d$ and
$\bigcap_{i\in I} C_i\neq\emptyset$ for each $I\subset\{1,\dots,m\}$
with $|I|\le d+1$, then the intersection of $C_1,\dots,C_m$ is non-empty.
\end{corollary}

\begin{proof} If the conclusion were false, there would be a least $r > d$ such that some $r$-simplex
of $\Delta_m$ was not contained in $\NC$, and this would provide an empty simplex.
\end{proof}

\subsection{Joins and spheres}
The join $K\ast L$ of two simplicial
complexes $K$ and $L$ is a simplicial complex whose
vertex set $(K\ast L)^{(0)}$ is the disjoint union of $K^{(0)}$
and $L^{(0)}$; for each $r$-simplex $[u_0,\dots,u_r]$
 in $K$ and $s$-simplex $[v_0,\dots,v_s]$ in $L$, there is
an $(r+s+1)$-simplex 
$[u_0,\dots,u_r,v_0,\dots,v_s]$ in $K\ast L$. Of particular importance
for us is the observation that $S^1:=\S^0\ast\S^0$ is, topologically,
a 1-sphere (more precisely a graph with four edges and four
vertices of valence 2), and that if one defines $S^n$ iteratively
by $S^n := S^{n-1}\ast\S^0$, then $S^n$ is homeomorphic
to the $n$-sphere $\S^n$.

The following special case of Theorem \ref{t:helly} will
be useful in the sequel.

\begin{corollary}\label{prods}
Let $C_{i0}, C_{i1}$ be closed
convex subsets of a convex metric space $X$,
with $i=0,\dots,d$. If $\dim X\le d$ and
 $C_{i\epsilon}\cap C_{k\delta}$ is non-empty for all $\epsilon,\delta \in\{0,1\}$
whenever $i\neq k$,  then $C_{i0}\cap C_{i1}$
is non-empty for some $i\in\{0,\dots,d\}$.
\end{corollary}

\begin{proof} If the conclusion of the 
corollary failed then
the nerve of the collection $\C=\{C_{ij}\}_{i,j}$ would be a
join of $(d+1)$ 0-spheres, i.e.~$\NC\approx\S^d$.
\end{proof}

\section{The $\Delta_n$ criterion}\label{s:ftp}
We want to apply the preceding results to 
collections of fixed-point sets. Our basic goal is to promote the existence of
fixed points for collections of subgroups in a fixed group 
to the existence of fixed points for the ambient group. 

The following basic example of how to do this is now well known and has been frequently used.

 \begin{prop}[$\Delta_n$ Criterion]\label{basicHelly}\label{p:Delta}
  Let $\G$ be a group that is
 generated by the union of finitely many subsets $A_i$
 and let $X$ be a complete \cat space of dimension $\le d$ on which $\G$
 acts by isometries. If the subgroup generated by the union of each
 $d+1$ of the sets $A_i$ has a fixed point in $X$, then
 $\G$ has a fixed point.
 \end{prop}
 
 \begin{proof} Apply Corollary \ref{metric:helly} to the fixed point sets of the $A_i$.
 \end{proof}
 
\begin{cor}[$\Delta_n$ Torsion]\label{c:Dtor}
Let $\G$ be a group
generated by the union of the subsets $A_1,\dots,A_n$. Let $H_J<\G$ be the subgroup
generated by $\{A_j\mid j\in J\}$. If $H_J$ is finite whenever $|J|\le d+1$, then 
$\fd(\G)\ge d$.
\end{cor}

\begin{proof}
Finite groups of isometries always have fixed points (Corollary \ref{c:finite}).
\end{proof}

\subsection*{Example: Simplices of finite groups} An {\em{$n$-simplex of groups}}
is a contravariant
functor $\mathcal S$
 from the poset of non-empty faces of an $n$-simplex,
ordered by inclusion, to the category of groups and monomorphisms;
the resulting diagram of groups is required to commute; see \cite{BH} p.377.
The {\em{fundamental group}} $\pi_1\mathcal S$ is the direct limit of the resulting diagram in the category of
groups. $\mathcal S$ is said to be {\em{gallery-connected}} if
the images of the groups associated to the codimension-one faces
together generate $\pi_1\mathcal S$.

The following special case of the $\Delta_n$-criterion was investigated by Angela Barnhill \cite{barnhill}.  
 
\begin{corollary}[Simplices of groups]\label{simplex}
If $\G$ is the fundamental group of a gallery-connected
$n$-simplex of finite groups, then $\fd(\G)\ge n-1$. 
\end{corollary}

\begin{proof} Apply Corollary \ref{c:Dtor} with the groups $\mathcal S_{\sigma_i}$
associated to codimension-1 faces $\sigma_i$ in the role of the
$A_i$. If $|J|\le n$, then $H_J$ is contained in the finite group $\mathcal S_{\tau}$,
where $\tau = \cap_{i\in J}\sigma_i$.
\end{proof}

If an $n$-simplex of finite groups supports a metric
of non-positive curvature (in the
sense of \cite{BH} p.388), then its fundamental group acts properly
and cocompactly by isometries on an $n$-dimensional
\cat space with fundamental domain a single simplex.
In \cite{JS}, Januszkiewicz and
Swiatkowski constructed,  for all $n>0$, examples of {\em{hyperbolic groups}} that
arise in this way. Thus:

\begin{prop}\label{p:hyp} For every positive integer  $n$, there exist hyperbolic groups $\Lambda_n$ with $\fd(\Lambda_n)=n$.
\end{prop}

In \cite{gang6} we explain how these examples
can be used to construct an infinite, finitely generated group
that cannot act without a fixed point on any complete,
finite dimensional acyclic space.

\section{The Product Lemma and Bootstrapping}\label{s:prods} \label{s:bootstrap}

In the previous section we saw how Corollary \ref{metric:helly} led to a fixed-point criterion. In this
section we shall see how other special cases of Theorem \ref{t:helly} lead to criteria that are more widely
applicable.  

\begin{lemma}\label{p:join} Let $X$ be a complete {\rm{CAT}}$(0)$
space and let $S_1,\dots,S_\ell\subseteq \isom(X)$ be subsets
such that $[s_i,s_j]=1$ for all $s_i\in S_i, \, s_j\in S_j\ (i\neq j)$.
If $\Ne_i$ is the nerve of the family $\C_i=(\fix(s_i)\mid s_i\in S_i)$,
then the nerve $\Ne$ of $\C_1\sqcup\dots\sqcup\C_\ell$ is the join
$\Ne_1\ast\dots\ast\Ne_\ell$.
\end{lemma}

\begin{proof} It is clear that $\Ne$ is contained in the join of
the $\Ne_i$; we must argue  that the converse is true, i.e.~for each
$\ell$-tuple  $\underline{\sigma}$ of simplices  $\sigma_i <\Ne_i\ (i=1,\dots,\ell)$, there is
a simplex in $\Ne$   that is the join of the $\sigma_i$. 
Let $H_i<\isom(X)$ be the subgroup generated by the elements  of
$S_i$ indexing the vertices of $\sigma_i$ and note that $[H_i,H_j]=1$ if $i\neq j$.  The presence of $\sigma_i$
in $\Ne_i$ is equivalent to the statement that $\fix(H_i)$ is  non-empty.
Corollary \ref{c:commutes}
then tells us $\fix(\cup_i H_i)=\cap_i\fix(H_i)$ is non-empty, as required.
\end{proof}

\begin{prop}[Bootstrap Lemma]\label{l:bootstrap}
Let $k_1,\dots,k_n$ be positive
integers and let $X$ be a complete \cat space of dimension less than $k_1+\dots+k_n$.
Let $S_1,\dots,S_n \subseteq\isom(X)$ be subsets with $[s_i,s_j]=1$
for all $s_i\in S_i$ and $s_j\in S_j\ (i\neq j)$.

If, for $i=1,\dots,n$, the subgroup generated by each $k_i$-element subset of $S_i$ has a 
fixed point in $X$, then for some $i$ every 
finite subset of $S_i$ has a common fixed point.
\end{prop} 

\begin{proof} Suppose that the conclusion
of the proposition were false.
Then for $i=1,\dots,n$ there would be a smallest
integer $k_i'\ge k_i$ such that some
$(k_i'+1)$-element subset 
$T_i=\{s_{i,1},\dots, s_{i,k_i'+1}\}$ in $S_i$  did not have a fixed point.  

Since any $k_i'$
elements of $T_i$ have a common fixed point, 
the nerve of the family $\C_i=(\fix(s_{i,j})\mid j=1,\dots, k_i'+1)$
would be the boundary of a $k_i'$-simplex $\partial\Delta_{k_i'}$. Hence, by
Lemma \ref{p:join}, the nerve of  $\C_1\sqcup\dots\sqcup \C_n$  
would be the join  $\partial\Delta_{k_1'}\ast
\dots\ast\partial\Delta_{k_n'}$. But this
contradicts Theorem \ref{metric:helly}, because 
the realisation of this join
is homeomorphic to a sphere of dimension $(\sum_{i=1}^n k_i') -1\ge\dim X$.
\end{proof}

\begin{cor}[Product Lemma]\label{c:prodLemma} For $d>0$,
let $\G=\G_1\times\dots\times\G_d$.
If for $i=1,\dots,d$ the group $\G_i$ is generated
by the union of  finitely many finitely generated subgroups $H_{ij}$, and
if $\fd(H_{ij})\ge d-1$ for all $i,j$, then whenever $\G$ acts 
by isometries on a complete \cat space of dimension
less than $d$, at least one of the factors $\G_i$ has a fixed
point.
\end{cor}
 
 \begin{proof}
 Let $S_i=\bigcup_j H_{ij}$ and take $k_i=1$. 
  \end{proof}

The following special case of Corollary \ref{c:prodLemma} will be particularly useful.

\begin{cor}(= Proposition \ref{p:prodLemma}) If each of the
groups $\G_1,\dots,\G_d$ has a finite
generating set consisting of elements of 
finite order, then at least one of the $\G_i$ has a fixed point whenever 
$\G_1\times\dots\times\G_d$ acts by isometries of a complete \cat space of dimension
less than $d$.
\end{cor}

By taking $k_i=2$ in Proposition \ref{l:bootstrap} we obtain the following criterion.
Again, the case where all of the groups $ \<A_{ij}, A_{ik}\> $ are finite is already useful.
The reader can easily formulate the analogous statements with $k_i$ taking larger
constant values.

\begin{cor}[Filling Triples]\label{p:triple} Let $\G=\G_1\times\dots\times\G_d$ where
each group $\G_i$ is generated by the union of three subsets $A_{i1}\cup A_{i2}\cup
A_{i3}$ such that $\fd \<A_{ij}, A_{ik}\> \ge 2d-1$ for  $i=1,\dots,d$ and $j\neq k$. Then, whenever $\G$ acts
by isometries on a complete \cat space of dimension less than $2d$, one of the factors $\G_i$
has a fixed point.
\end{cor}

When applying the above results one has to wrestle with
the fact that the conclusion only provides a fixed point for one of the factors. 
A convenient way of gaining more control 
is to restrict attention to conjugate sets.

\begin{cor}[Conjugate Bootstrap]\label{c:conjug}\label{c:conjBS}
Let $k$ and $n$ be positive
integers and let $X$ be a complete \cat space of dimension less than $nk$.
Let $S_1,\dots,S_n$ be conjugates of a subset  
$S\subseteq\isom(X)$ with $[s_i,s_j]=1$
for all $s_i\in S_i$ and $s_j\in S_j\ (i\neq j)$.

If each $k$-element subset of $S$ has a 
fixed point in $X$, then so does each finite subset
of $S$ and of $S_1\cup\dots\cup S_n$.
\end{cor}     

\begin{proof} Proposition \ref{l:bootstrap} tells us that $S$ has a fixed point, and from Corollary
\ref{c:commutes} it follows that $S_1\cup\dots\cup S_n$ does too.
\end{proof}

\subsection{Wreath products}

We remind the reader that the (restricted) {\em{wreath product}}
$B\wr T$  is the semidirect product $(\oplus_{t\in T}B_t)
\rtimes T$, where there are fixed isomorphisms $B\cong B_t$  and the action of $T$
permutes the indices $t$ by left multiplication. {\em{Permutational wreath products}}
$B\wr_\rho T$
are defined similarly but with arbitrary index sets $I$ and a prescribed action
$\rho: T\to \sym(I)$.

We write $C_n$ to denote the cyclic group of order $n$.
\begin{cor} If $\G$ is generated by the union of 
finitely many subgroups $H_j$ with $\fd(H_j)\ge d-1$, then
$\fd(\G\wr C_d)\ge d-1$.
\end{cor}

\begin{proof}  By applying Corollary \ref{c:conjBS} with $k=1$ and $n=d$, we see that  
$\fd(\oplus_{t\in C_d} \G)\ge d-1$, and Corollary \ref{fixFI} promotes this to $\G\wr C_d$.
\end{proof}
  
Again, we emphasise the case where $\G$ is torsion-generated.

\begin{cor}\label{c:propA+}
 If $\G$ is generated by a finite set of elements of finite order, then $\fd(\G\wr C_{d})\ge d-1$.
\end{cor}

The same argument applies to permutational wreath products.

\begin{cor} Let $d>0$ be an integer,  $G$   a finite
group, and  $\rho:G\to \sym(d)$  a transitive
permutation representation.
If $\G$ is generated by the union of 
finitely many subgroups $H_j$ with $\fd(H_j)\ge d-1$
then $\fd(\G\wr_\rho G)\ge d-1$.
\end{cor}

\subsection*{Bieberbach groups}

\begin{proposition}\label{p:bieb}  There exist groups $(\G_n)_{n\in\N}$
such that $\G_n$ acts properly and cocompactly by
isometries on $\E^n$ but cannot act without a fixed
point on any complete \cat space of dimension less than $n$.
\end{proposition}
  
\begin{proof} Let  $P_n$ be the  group generated
by reflections in the sides of a cube in  $\E^n$ and let the symmetric group
$\sym(n)$ permute the
coordinate directions of the cube. Then  $P_n$ is a 
direct product of $n$ infinite dihedral groups $D_\infty$ and
 $\G_n = P_n\rtimes \sym(n)$ is a group of the type
described in the preceding corollary. (Alternatively, one could take $D_\infty\wr C_n$.) 
\end{proof}

\section{A general scheme: Ample Duplication}\label{s:ample}

For the convenience of the reader, we recall from the  introduction
the statement of the Ample Duplication Criterion. Note that the 
definition   refers to a generating set for  a {\em subgroup} $\Lambda<\G$
and all conjugates are taken in the ambient group $\G$.
There are no implicit constraints on the function $f(k)$ in this definition,
but the reader may want to keep the duplication functions from Proposition \ref{p:ample} or Theorem \ref{t:mcg} in mind as examples.

\begin{definition}\label{d:ample}
Given a group $\G$ and positive
integers $d$ and $k_0$, we say that a finite generating set
$\A$ for a subgroup $\Lambda<\G$ has {\em{ample duplication
for dimension $d$, with base $k_0$,}} if there is a function  
$f:\mathbb{N}\to \mathbb{N}$ such that
the following conditions both hold:
\begin{enumerate}
\item  Each subset $S\subseteq\A$ of cardinality $|S|> k_0$ can either be written
as a disjoint union $S=S_1\sqcup S_2$ where the $S_i$ are non-empty
and $\langle S_1\rangle$ normalizes $\langle S_2\rangle$,
or else there are at least $f(|S|)$ commuting conjugates of 
$\langle S\rangle$ in $\G$;
\item $d < (k-1)\,f(k)$ for $k=k_0+1,\dots,\min\{d+1,\, |\A|\}$. 
\end{enumerate}
When these conditions hold, $f(n)$ is said to be an {\em ample duplication function}.
\end{definition}

\begin{theorem}[Ample Duplication Criterion]\label{t:scheme} 
Let $\G$ be a group acting by isometries on a complete \cat space $X$ of dimension
at most $d$, and let $\Lambda<\G$ be a subgroup with
a finite generating set $\A$ that has ample duplication
for dimension $d$, with base $k_0$. If  $\<S\>$ has a fixed point for every $S\subseteq\A$ with $|S|\le k_0$,
then $\Lambda$ has a fixed point in $X$.
\end{theorem}

\begin{proof} We shall argue by induction on
$|S|$ to show that $\langle S\rangle$ has a fixed
point for every $S\subseteq\A$. We have assumed this is true for $|S|\le k_0$. 

Suppose now that $|S|=k>k_0$ and that smaller subsets of $\A$
all have common fixed points. If $S=S_1\sqcup S_2$, as in
condition (1), then  the common fixed points of 
 $\langle S_1\rangle$
and $\langle S_2\rangle$ provided by Proposition \ref{p:norm} are fixed points for $\langle S\rangle$.
If not, then condition (1) provides
 $l:=f(k)$ commuting conjugates of $\langle S\rangle$,
 say $\Sigma_1,\dots,\Sigma_l$, and the Conjugate Bootstrap (Corollary \ref{c:conjBS})
 provides a fixed point for $\<S\>$ provided the inequality $(k-1)\, f(k) > d$ holds, which
 it does by condition (2), unless $k\ge d+2$, in which
 case our inductive hypothesis tells us that every $(d+1)$-element
 subset of  $\A$ has a fixed point in $X$, and Proposition \ref{p:Delta} applies (with the $A_i$ as singletons).
 \end{proof}

\subsection{Mapping Class Groups} Building on classical work of Max Dehn, Raymond Lickorish \cite{ray} proved that the mapping class group $\modg$ of 
a closed orientable surface of genus $g\ge 2$ is generated by the Dehn twists in $3g-1$ 
simple closed curves. In \cite{mb:dim} the following proposition is proved via
a lengthy analysis of the subsurfaces supporting  subsets of these generators.

\begin{theorem}{\cite{mb:dim}}\label{t:mcg}
The Lickorish generators of the mapping class
group $\modg$ have ample duplication for dimension $g-1$, with base $1$
and duplication function
\begin{equation*}
f(k) =
\begin{cases} \lfloor 2g/k\rfloor & \text{$k$ even,}
\\
\lfloor 2(g-1)/(k-1)\rfloor  &\text{$k$ odd.}
\end{cases}
\end{equation*}
 \end{theorem}

The theorem includes the assertion that the displayed function $f(n)$
satisfies the second condition in the definition of ample duplication. 
This is an elementary but  instructive exercise.

It is proved in \cite{mb:bill} that if $g\ge 3$, then Dehn twists
have fixed points whenever $\modg$ acts by semsimple isometries on a complete
${\rm{CAT}}(0)$ space, and special considerations apply for $g=2$. This allows
one to deduce the following consequence of the above theorem.

\begin{corollary}\cite{mb:dim}\label{c:mcg} Whenever $\modg$ acts by semisimple isometries on a
complete \cat space of dimension less than $g$, it fixes a point.
\end{corollary}

\subsection{Braid Groups and Base Variation} The braid group on $m$ strings has the well known
presentation
$$
B_m = \< \sigma_1,\dots,\sigma_{m-1} \mid
\sigma_i\sigma_{i+1}\sigma_i=\sigma_{i+1}\sigma_i\sigma_{i+1},
\ [\sigma_i,\sigma_j]= 1 \text{ if $|i-j|>1$}\>.
$$
It is natural to define $\sigma_m =\s_1\cdots\s_{m-2}\s_{m-1}\s_{m-2}^{-1}\cdots\s_1^{-1}$; the relation
$[\s_i,\s_j]=1$ if $|i-j|\neq 1 \mod m$ then holds. In
the picture of the braid group as the mapping class group of an $m$-punctured disc, including $\s_m$ corresponds to arranging the
punctures in a circle rather than a line.

The braid group on at most $7$ strings acts properly and cocompactly by isometries on a polyhedral
 CAT$(0)$ space \cite{brady}, \cite{HKS}, \cite{braid7}; it is unknown if braid groups
on more strings have similar actions. The following proposition
is not sharp but we include it because it concerns groups of great geometric interest
and because its proof provides a faithful illustration of how one uses
the Ample Duplication Criterion. It also serves as a warm-up for the more complicated arguments
that apply to $\aut$. A feature of particular note   is 
 that as one increases the level of the base
from which ample duplication is required, the dimension of the spaces for which
one obtains fixed points rises accordingly.

In the course of the proof we shall need the following
technical lemma (cf.~\cite{mb:dim} Lemma 4.2).

\begin{lemma}\label{l:count} For positive integers $n$ and $k$, define $g_n(k) = (k-1) \lfloor n/(k+1)\rfloor$. Then,
\begin{enumerate}
\item $g_n(2) \le g_n(k)$ for  $2\le k \le n-1$ with equality
if and only if $(n,k)\in \{(6,3),\, (7,3),\, (9,4)\}$.
\item $g_n(3)\le g_n(k)+1$  for $3\le k \le n-1$ with equality
if and only if $k$ is even and $n\in \{2k,\, 2k+1\}$. 
\end{enumerate}
\end{lemma}

The proof of this lemma is entirely elementary, but it is instructive, as is the 
plot of small values of $g_n(k)$ (table below).
The pain caused by the fact that $g_n(k)$ is not a monotone function of $k$ is a hallmark of
many proofs in this area.  The circled numbers highlight the failures of monotonicity; there
will be a total of $\frac{1}{2}(k-1)(k-2)-1$ circled entries in column $k$.

\begin{table}[ht]
\begin{center} 
\caption{Values $g_n(k)$ with failures of monotonicity circled. }
\begin{tabular}{r|rrrrrrrrrrrrrrrrrrr}
  \hline 
  & $k=2$ &  3 & 4 & 5 & 6 & 7 & 8 & 9 & 10 & 11 & 12 & 13 & 14 & 15 & 16  \\
  \hline
  $n=3$    & 1 & 0 & 0 & 0 & 0 & 0 & 0 & 0 & 0 & 0 & 0 & 0 &  0 & 0 & 0  \\
  4   & 1 & 2 & 0 & 0 & 0 & 0 & 0 & 0 & 0 & 0 & 0 & 0 & 0 & 0 &  0 \\
  5   & 1 & 2 & 3 & 0 & 0 & 0 & 0 & 0 & 0 & 0 & 0 & 0 & 0 & 0 &  0 \\ 
  6   & 2 & 2 & 3 & 4 & 0 & 0 & 0 & 0 & 0 & 0 & 0 & 0 & 0 & 0 &  0  \\ 
  7   & 2 & 2 & 3 &  4 &  5 & 0 & 0 & 0 & 0 & 0 & 0 & 0 & 0 & 0 &  0  \\
  8   & 2 & 4 & {\textcircled {3}} & 4 & 5 & 6  & 0 & 0 & 0 & 0 & 0 & 0 & 0 & 0 &  0 \\
  9   & 3 & 4 &  {\textcircled {3}} & 4 & 5 & 6  & 7 & 0 & 0 & 0 & 0 & 0 & 0 & 0 &  0 \\ 
  10  & 3 & 4 &  6 & {\textcircled {4}} & 5 & 6  & 7 &  8 & 0 & 0 & 0 & 0 & 0 & 0 &  0 \\ 
  11  & 3 & 4 &  6 & {\textcircled {4}} & 5 & 6  & 7 &  8 & 9 & 0 & 0 & 0 & 0 & 0 &  0   \\
  12  & 4 & 6 & 6 & 8 & {\textcircled {5}} & 6  & 7 &  8 & 9 & 10 & 0 & 0 & 0 & 0 &  0  \\
  13  & 4 & 6 & 6 & 8 & {\textcircled {5}} & 6  & 7 &  8 & 9 & 10 & 11 & 0 & 0 & 0 &  0 \\ 
  14  & 4 & 6 & 6 & 8 & 10 & {\textcircled {6}} & 7 &  8 & 9 & 10 & 11 & 12 & 0 & 0 &  0  \\ 
  15  & 5 & 6 & 9 & {\textcircled {8}} & 10 & {\textcircled {6}}  & 7 &  8 & 9 & 10 & 11 & 12 & 13 & 0 &  0    \\ 
  16  & 5 & 8 & 9 & {\textcircled {8}} & 10 & 12  & {\textcircled {7}} &  8 & 9 & 10 & 11 & 12 & 13 & 14 &  0  \\ 
  17  & 5 & 8 & 9 & {\textcircled {8}} & 10 & 12  & {\textcircled {7}} &  8 & 9 & 10 & 11 & 12 & 13 & 14 &  15  \\ 
   \hline
\end{tabular}
\end{center}
\end{table}

\begin{prop}\label{p:B_m} The generating set $\underline\sigma=\{\s_1,\dots,\s_m\}$
 for $B_m$ has ample duplication for dimension
$\lfloor m/3\rfloor-1$ with base $1$ and duplication function
$$
f_m(k) = \lfloor m/(k+1)\rfloor.
$$
The same duplication function is ample for dimension $2\lfloor m/4\rfloor-2$ with base $2$.
It is also ample for dimension $2\lfloor m/4\rfloor-1$ with base $2$ if $m\equiv 2$ or $3 \mod 4$.
\end{prop}

\begin{proof} Given a proper subset $S_I=\{\sigma_i\mid i\in I\}$ of $\underline\s$ with $|I|=k\ge 2$,
we may conjugate by a power of $\s_1\cdots\s_{m-1}$ to assume that $\s_m\not\in I$. Then,
either $I$ can be written as a disjoint union $I_1\sqcup I_2$ with $\max I_1 < \min I_2-1$, 
or else $I$ is an interval $[i_0,i_0+k-1]\cap \N$. In the first case, $\<S_{I_1}\>$
commutes with $\<S_{I_2}\>$, and in the second case we have $f_m(k):=\lfloor m/(k+1)\rfloor$
commuting conjugates of $\<S_I\>$ in $B_m$,  
namely the subgroups generated by
$J_0,\dots, J_{f(k)-1}$ where $J_{r}=\<\sigma_{r(k+1)+1},\dots,\sigma_{r(k+1)+k}\>$.

To see why this is true, note that we are dealing with the case where, in the standard
braid-diagram representation of the braid group,  $\<S_I\>$ is supported on a block $k+1$
strings; in $J_i$, we have translated the support to the block of strings beginning with string $(k+1)i +1$,
and successive blocks sit next to each other (but do not overlap, ensuring that the $J_i$ commute).
The number of disjoint blocks that we can fit in is $\lfloor m/(k+1)\rfloor$, which is $f_m(k)$.

For the assertion in the first sentence of the proposition,
we take $k_0=1$ in the Ample Duplication Criterion and are required to prove that $ \lfloor m/3\rfloor \le (k-1) f_m(k)$
for $k=2,\dots,\lfloor m/3\rfloor$; equivalently, $g_m(2)\le g_m(k)$.
This is covered by Lemma \ref{l:count}(1).

For the second assertion, we take $k_0=2$ and the required bound is
$ 2 \lfloor m/4\rfloor-1 \le g_m(k)$, which is covered by Lemma \ref{l:count}(2). 

For the third assertion, we need the inequality $g_m(3) - 1< g_m(k)$ for
$k=3,\dots,2\lfloor m/4\rfloor$. This is valid for $m\le 7$ but if $m=8$ or $9$ then $k=4$ causes a problem.
For $m=10$ or $11$, there is no problem up to $k=4=2\lfloor m/4\rfloor$, so the equality is valid;
likewise it is valid whenever $m\equiv 2$ or $3 \mod 4$. But
for $m=12$ or $13$ the inequality fails when $k=6$, and in general it fails when $m$ is congruent to $0$
or $1 \mod 4$ and $k = 2\lfloor m/4\rfloor$. 
\end{proof}

\begin{corollary}\label{c:braid}
Suppose that $B_m$ acts by isometries on a complete \cat space $X$.
\begin{enumerate}
\item If each generator $\sigma_i$ has a fixed point and 
$\dim X< \lfloor m/3\rfloor$, then $B_m$ has a fixed point.
\item If one of the subgroups $\<\sigma_i,\sigma_{i+1}\>\cong B_3$
 has a fixed point and $\dim X < 2\lfloor m/4\rfloor-\delta_m$, then $B_m$ has a fixed point,
 where $\delta_m = 0$  if $m\equiv 2$ or $3 \mod 4$, and  $\delta_m = 1$  if $m\equiv 0$ or $1\mod 4$.
\end{enumerate}
\end{corollary}

\section{Generators and subgroups in the automorphism groups of free groups}

Our most serious application of the Ample Duplication Criterion is to automorphism groups of free groups. Before
proceeding to this, we need to recall some basic facts about how these groups can be generated.

We fix a basis $\{x_1,\dots,x_n\}$ for the free group of rank $n$.
If $n\ge 3$, then \autn has a unique subgroup of index $2$, which is denoted by $\sautn$.
J.~Nielsen \cite{nielsen} proved that $\sautn$ is generated by the  {\em Nielsen automorphisms}
$\l_{ij}$ and $\rho_{ij}$, which are defined as follows: $$\l_{ij}(x_i)=x_jx_i\  \  \hbox{and} \  \   
\l_{ij}(x_k)=x_k \  \  \hbox{if}\  \  k\neq i$$
and $$\r_{ij}(x_i)=x_ix_j\  \  \hbox{and} \  \   
\r_{ij}(x_k)=x_k \  \  \hbox{if}\  \  k\neq i.$$ 
To generate the whole of \autn one can add one of the automorphisms $\e_i$, where
$$\e_i(x_i)=x_i^{-1} \  \  \hbox{and} \  \    \e_i(x_k)=x_k  \  \  \hbox{if}\  \  k\neq i.$$ 

By making repeated use of the relations\footnote{for the left action of ${\rm{Aut}}(F_n)$ with the
commutator convention $[\alpha, \beta] = \alpha^{-1}\beta^{-1}\alpha\beta$.}
$[\l_{jk},\l_{ij}]=\l_{ik}$, one sees that the
generators $\l_{ik}$ with $|i-k|\neq 1\mod n$ are
unnecessary. Likewise, one can dispense
with the generators $\rho_{ik}$  with $|i-k|\neq 1\mod n$. Thus we arrive at:

\begin{lemma}\label{l:NiGen}
${\rm{SAut}}(F_n)$ is generated by the union of the $n$
sets  (indices
$\mod n$)
$$\Ni_i=\{\l_{i,i-1},\,\rho_{i,i-1}\}.$$
\end{lemma} 

The bounds established in the following proposition 
will be superceded in the next section, but we include them here because the
arguments are so much easier.

\begin{prop}\label{p:ample} The Nielsen generators for $\sautn$ have ample duplication for dimension
$\lfloor n/3\rfloor-1$ with base $1$ and duplication function
$$
h_n(k) = \lfloor {n}/{(k+1)}\rfloor.
$$
The same duplication function is ample for dimension $2\lfloor n/4\rfloor-1$ with base $2$,
if  $m\equiv 2$ or $3 \mod 4$,
and is ample for dimension $2\lfloor m/4\rfloor-2$ with base $2$, if $m\equiv 0$ or $1 \mod 4$.
\end{prop}

\begin{proof} Given a set $S$ of $k$ Nielsen transformations $\nu_{ij}$, with $2\le k\le \lfloor n/3\rfloor$,
we may conjugate by a suitable permutation of the basis $\{x_1,\dots,x_n\}$ to assume that
 $\{i,j\}\neq\{1,n\}$ for $\nu_{ij}\in S$.
If the set of indices $i,j$ arising in $S$ breaks into two disjoint, non-empty sets, then we get a decomposition
$S=S_1\sqcup S_2$ such that $\<S_1\>$ commutes with $\<S_2\>$. If not, then
the indices form an interval of length at most $k+1$.
In this case, $\sautn$ contains  $h_n(k)=\lfloor n/(k+1)\rfloor$
commuting conjugates of $\<S\>$, namely the conjugates of $\<S\>$ by the
permutations of $\{x_1,\dots,x_n\}$ that shift the indices (mod $n$) $x_i\mapsto x_{i+r(k+1)}$
for $r=0,\dots, h_n(k)-1$.  From this point, the proof of Proposition \ref{p:B_m} applies. 
\end{proof}

\subsection{Generating \autn by torsion elements}

The following result was used in \cite{mb:zies} to give a short proof of 
the fact that if $n\ge 3$ then $\aut$ has property ${\rm{F}}\R$.

\begin{prop}\label{p:A123} For $n\ge 3$, there exist sets $A_1,A_2,A_3\subset\aut$
such that $\langle A_i,\, A_j\rangle$ is finite for $i,j=1,2,3$
but $A_1\cup A_2\cup A_3$ generates $\aut$.  If $n\ge 4$ then  $\saut$ satisfies the same condition. 
\end{prop}

The exact nature of the sets $A_i$ will not be important
here, but we include a brief description of them to show that they are not complicated.
Let $\sym(n) <\aut$ be the group generated by permutations of our fixed basis
$\{x_1,\dots,x_n\}$. We write $(i\ j)$ to denote the involution that interchanges $x_i$ and $x_j$.
The involution $\e_i$ was defined earlier.
Let $W_n\cong C_2^n\rtimes \sym(n)$ be the group generated by $\sym(n)$ and the elements $\e_i$.
Let $\sym(n-2) < \sym(n)$ and $W_{n-2}< W_n$ be the subgroups corresponding to the sub-basis 
$\{x_3,...,x_n\}$, let  $\theta=\r_{12}\circ\e_2,\ \tau=(2\ 3)\circ \e_1$ and $\eta=(1\ 2)\circ\e_1\circ\e_2$. Then
define
$A_1 =\{\e_n,\eta\}\cup \sym(n-2),\  A_2 =\{\theta\},$ and $ A_3 =\{\tau\}$.  See \cite{mb:zies} for details.

\subsection{Dihedral Generators for \autn}

There are many different ways of generating $\aut$ by elements of order $2$. One way is to note that
every Nielsen transformation is contained in an infinite dihedral group, namely $L_{ij}:=\<\l_{ij},\e_j\>$
or $R_{ij}:=\<\l_{ij},\e_j\>$.  Note that all such subgroups are conjugate in $\aut$,
with $\e_i$ conjugating $R_{ij}$ to $L_{ij}$ and permutations of the standard basis conjugating
the different $R_{ij}$ to each other.

\begin{lemma}\label{l:dihedral} If $n=3m$ (resp. $3m+2$) 
then \autn contains the direct product of $2m$ (resp. $2m+1$)
$\infty$-dihedral groups, each of which contains a Nielsen transformation. Moreover, these dihedral
groups are all conjugate.
\end{lemma}

\begin{proof} If $n=3m$ we take the product of $R_{3i+1,3i+2}$ and $L_{3i+1,3i+3}$ for $i=0,\dots,(m-1)$,
and if $n=3m+2$ we can add $R_{3m+1,3m+2}$.  
\end{proof}

\begin{remark}\label{r:dihedrals}
$\sautn$ contains a 
direct product of $(n-1)$ infinite dihedral
groups, namely $D_i=\langle\lambda_{i1}\rho_{i1}^{-1},\ \e_i\e_1\rangle$
with $i=2,\dots,n$. But these $D_i$ have finite image in ${\rm{GL}}(n,\Z)$, so their
product does not contain any Nielsen transformations. 
\end{remark}

\subsection{The column subgroups $M_n(i)$ of $\sautn$}

In $\sln$ the elementary matrices with off-diagonal entries in positions $(1,n),\dots,(n-1,n)$
generate a free abelian group of rank $(n-1)$. These generators lift to
Nielsen transformations $\l_{nj}\in\sautn$ that generate a non-abelian free group of
rank $(n-1)$ which we
denote  by $M_n(n-1)$. More generally, we define {\em column subgroups} $M_j(m)\cong F_m$ as follows.

\begin{definition} For integers $1\le m<j\le n$ define $M_j(m)<{\rm{SAut}}(F_n)$ to be the subgroup generated by 
$\{\l_{j1},\dots,\l_{jm}\}$, and $\-M_j(m)$ to be the subgroup generated
by $\{\r_{j1},\dots,\r_{jm}\}$.   
\end{definition}
  
\begin{lemma}\label{l:mu} For all positive integers $m<n$, there is a family of $2(n-m)$ commuting
conjugates of $M_n(m)$ in $\sautn$.
\end{lemma}

\begin{proof} Let $\zeta$ be the automorphism that fixes $x_1,\dots,x_m$
and cyclically permutes $x_{m+1},\dots,x_n$ (composed with $\e_1$ if $n-m$ is odd).
For $i=0,\dots, n-m-1$, the conjugates
of $M_n(m)$ by $\zeta^i$, which are all of the form $M_j(m)$, pairwise commute.
The conjugate of $M_j(m)$ by $\e_1\e_j$ is $\-M_j(m)$, and 
for $j,j'\in\{m+1,\dots, n\}$ the subgroups $ M_j(m), \,  \-M_{j}(m),\,  M_{j'}(m),\,  \-M_{j'}(m)$
all commute with each other.  
\end{proof}

\begin{corollary} \label{c:getMn} The generators $\{\l_{n1},\dots\l_{n,n-1}\}$
for $M_n(n-1)<{\rm{SAut}}(F_n)$ have ample duplication for dimension $2n-5$ with base $1$
and duplication function
$$
f(m)=2(n-m).
$$
\end{corollary}

\begin{proof} Each $m$-element subset of the given generators
 generates a conjugate of $M_n(m)$, and Lemma \ref{l:mu} provides $2(n-m)$ commuting
conjugates of this, so it  suffices to check that $2n-4\le 2(m-1)(n-m)$ for $m=2,\dots,n-1$, which one
can do by noting that the parabola $y=2(x-1)(\nu -x)$ meets the horizontal line $y=2\nu-4$
at $x=2$ and $x=\nu -1$.
\end{proof}

\begin{prop}\label{p:getMn}
Suppose that $\sautn$ acts by isometries on
a complete \cat space $X$ of dimension
$d\le 2n-5$. If a Nielsen transformation has a fixed point in $X$,
then so does $M_n(n-1)\times \-M_n(n-1)$.
\end{prop}

\begin{proof} Corollary \ref{c:getMn}  allows us to  apply the Ample Duplication Criterion to $M_n(n-1)$
to conclude that $M_n(n-1)$ has a fixed point in $X$. It follows that
 its conjugate $\-M_n(n-1)$ does too. These subgroups commute, so
their product $M_n(n-1)\times \-M_n(n-1)$ also has a fixed point.
\end{proof}

\section{Fixed Points for Nielsen-Elliptic Actions}

The purpose of this section is to prove the following theorem and a variant concerning actions on Hadamard manifolds (Theorem \ref{t:hadamard}). Lemma \ref{l:fi} tells us that if $\aut$ is acting by isometries on a complete \cat space
and $\sautn$ has a fixed point, then so does $\aut$, so to obtain the sharpest results we concentrate on $\sautn$.

\begin{theorem} \label{t:n-1}
Suppose that $\sautn$ acts by isometries on a complete \cat space $X$ of dimension less than $n-1$.
If a Nielsen transformation has a fixed point in $X$, then so does $\sautn$.
\end{theorem}

This result reduces the proof of Theorem \ref{i:main}
to the task of forcing Nielsen generators to have fixed points, which we pursue in the next section. However,
Theorem \ref{t:n-1} is also of interest in its own right, as we shall now explain.  

\begin{cor}\label{c:ss} Whenever $\sautn$ acts by semisimple isometries on a complete \cat space of dimension
less than $n-1$, it has a fixed point.
\end{cor}

\begin{proof} If $n\le 2$ the theorem is vacuous, and if $n=3$ it is the statement that $\sautn$ has property
${\rm{F}}\R$, which was proved by Culler and Vogtmann \cite{CV}. For $n\ge 4$, it is proved in \cite{mb:rddh},
using the structure of centralisers, 
that Nielsen transformations have fixed points in any semisimple action of $\sautn$ on a complete \cat space,
so Theorem \ref{t:n-1} applies.
\end{proof}

 The column subgroups $M_n(m)$ were introduced in the last section. In this
 section a prominent role will be played by $M=M_n(n-1)\times\-M_n(n-1)$.

A simple calculation shows:

\begin{lemma} If $2\le l\le n$ then $\Ni_{l}=\{\l_{l,l-1},\, \r_{l,l-1}\}$ normalizes $M$.
\end{lemma} 
 
\begin{lemma} \label{l:indN} Whenever
$\sautn$ acts by isometries on a complete \cat space $X$, if
$M<\sautn$ has a 
fixed point in $X$, then
the subgroup generated by  
the union of the sets $\Ni_i\ (i=2,\dots,n)$ has a fixed point in $X$. 
\end{lemma}

\begin{proof} Proceeding by induction, we shall argue that if
$k\le n-2$ and every $k$
of the sets $\Ni_i$ have a common fixed point in $X$, then
any $(k+1)$ of these sets do.

The sets $\Ni_i$ are all conjugate and $\Ni_n$ is contained in $M$,
so the base case $k=1$ is covered. For the inductive step,
we consider $k+1<n$ distinct 
sets of the form $\Ni_i$, indexed by $I\subset\{2,\dots,n\}$.
Either $I$ is the disjoint union of non-empty sets $I_1,\,I_2$
such that $|s-t|\neq 1$ for all $s\in I_1,\, t\in I_2$, or else
we may conjugate  in $\sautn$ to assume
that the sets are $\Ni_n,\Ni_{n-1},\dots,\Ni_{n-k}$. 

In the first case we know that each of the subgroups $H_j=
\langle \Ni_i \mid i\in I_i\rangle,\ j=1,2$,  has a fixed point,
because $|I_j|\le k$. And since these subgroups
commute, they have a common fixed point (Proposition
\ref{p:norm}), so we are done.

It remains to find a common fixed point for
$\Ni_n,\Ni_{n-1},\dots,\Ni_{n-k}$.  
By induction, $N:=\langle \Ni_{n-1},\dots,\Ni_{n-k}\rangle$ has
a fixed point in $X$, and so does $M$. As $N$
 normalizes $M$, they share a fixed point,  by Proposition \ref{p:norm}.   As $\Ni_n\subset M$, this completes the induction.
 \end{proof}

\noindent{\bf{Proof of Theorem \ref{t:n-1}} }
We are assuming that $\sautn$ is acting by isometries on a complete \cat space $X$
of dimension at most $n-2$ and that some  Nielsen transformation has a fixed point.
There is nothing to prove if $n\le 2$. If $n\ge 3$ then $n-2\le 2n-5$ and
 Proposition \ref{p:getMn} tells us that $M$ has a fixed point in $X$. Hence, by
Lemma \ref{l:indN}, the union of any $(n-1)$ of the sets $\Ni_1,\dots,\Ni_n$
has a common fixed point (such any two such unions are conjugate).
The union of the $\Ni_i$ generate $\sautn$ (Lemma \ref{l:NiGen}).
Thus we have a finite generating set such that every subset of cardinality $(n-1)$
has a fixed point. Since $\dim X < (n-1)$, Proposition \ref{basicHelly}  
applies and $\sautn$ has a fixed point.
\qed

\subsection*{Related Strategies}

The strategy of the proof used above has several variations of a
 general nature. We record one such variation but omit the proof.

\begin{prop}\label{p:norm2} Let $\A$ be a finite generating
set for a group $\G$ acting by isometries on a complete \cat space $X$
of dimension less than $d$. Let $M<\G$ be a subgroup and assume the following
conditions hold:
\begin{enumerate}
\item every element of $\A$ has a fixed point in $X$;
\item $M$ has a fixed point in $X$;
\item if $k\le d$ then for each $k$-element subset $S\subset\A$, either 
 $S=S_1\sqcup S_2$ where the $S_i$ are non-empty and $\<S_1\>$ normalises $\<S_2\>$, 
 or else  $\<S\>$ normalises $\g^{-1}M\g$ for some $\gamma\in\G$,
  and $\gamma^{-1}M\gamma\cap\A$ is not contained in $S$.
 \end{enumerate}
 Then $\G$ has a fixed point in $X$.
 \end{prop}

\subsection{Semisimple actions on Hadamard manifolds}

A {\em{Hadamard manifold}} is a simply connected manifold with a smooth, complete Riemannian metric with
non-positive sectional curvature.  Such manifolds are the most classical examples of \cat spaces.

\begin{theorem} \label{t:hadamard}  
$\sautn$ cannot act non-trivially by semisimple isometries on any Hadamard manifold
of dimension less than $2n-4$.
\end{theorem}

\begin{proof} The proof is by induction on $n$. If $n\le 2$, there is nothing
to prove. If $n=3$ or $4$,  then $2n-4 \le n$, and Bridson and Vogtmann \cite{BV} proved that
$\sautn$ cannot act non-trivially on any contractible manifold of dimension less than $n$
(even by homeomorphisms). 

Suppose now that $n\ge 5$. In this range, the Nielsen transformations have
fixed points whenever $\sautn$ acts by semisimple isometries on a \cat space \cite{mb:rddh}.
It follows from Proposition \ref{p:getMn} that $M = M_n(n-1)\times \-M_n(n-1)$ does too. 
Let $X$ be a Hadamard manifold
of dimension less than $2n-4$ on which $\sautn$ acts by isometries and let $Y$ be the
fixed-point set of $M$.  Because the action is by isometries,  $Y$ is a smooth, totally geodesic submanifold \cite[p.59]{kob};
in particular it is a Hadamard manifold. 

If $Y=X$ then we are done,   because  $M$
contains Nielsen transformations and the conjugates of any such transformation
 generate $\sautn$, so if $M$ was contained in the kernel of the action on $X$ then
 the action of  $\sautn$ would be trivial. 
 
Let $\delta$ be the codimension of $Y\subseteq X$.  
We will obtain a contradiction from the assumption that the action of $\sautn$
is non-trivial and $\delta\ge 1$.
 
 Observe that the normalizer of 
$M$ contains a natural copy of ${\rm{SAut}}(F_{n-1})$, consisting of the
automorphisms that fix the last element of our fixed free basis $\{x_1,\dots,x_n\}$
and leave $\<x_1,\dots,x_{n-1}\>$ invariant. Since it normalizes $M$, this subgroup
${\rm{SA}}_{n-1}\cong {\rm{SA}}_{n-1}$  leaves $Y$ invariant.  

If $\delta\ge 2$, then by induction
${\rm{SA}}_{n-1}$ acts trivially on $Y$.  The derivative of the action of ${\rm{SA}}_{n-1}$ at 
any point $p\in Y$ preserves the orthogonal complement of $T_pY$ in $T_pX$, giving
a representation  ${\rm{SA}}_{n-1}\to O(\delta,\R)$, which we shall prove
is trivial.

Potapchik and Rapinchuk \cite{PR} proved that in the range we are considering,  $n\ge 5$ and
$\delta < 2n-4$,  every representation   $\rho: {\rm{SAut}}(F_{n-1})\to {\rm{GL}}(\delta,\mathbb{C})$
factors through the standard representation ${\rm{SAut}}(F_{n-1})\to {\rm{SL}}(n-1,\Z)$ unless
the image of $\rho$ contains  a semidirect product 
$\Z^{n-1}\rtimes {\rm{SL}}(n-1,\Z)$.  Margulis superrigidity \cite{marg} tells us that 
${\rm{SL}}(n-1,\Z)$ cannot have infinite image in $O(\delta,\R)$,  so the image of ${\rm{SA}}_{n-1}\to O(\delta,\R)$
 must be a finite quotient of ${\rm{SL}}(n-1,\Z)$.  Every such quotient is
a finite extension of the simple group ${\rm{PSL}}(n-1,\Z/p)$ for some prime $p$. 
Lanazuri and Seitz \cite{LS} proved that for $N\ge 3$,  the minimal degree of a complex representation of 
${\rm{PSL}}(N,\Z/p)$ occurs when $p=2$,  
where the degree is $2^{N-1} -1$, and  
Kleidman and Liebeck \cite{KL} proved
that no finite extension of ${\rm{PSL}}(N,\Z/p)$ has a faithful representation of lesser degree. 
In our situation, $n\ge 5$,  so $N=n-1\ge 4$ and $\delta \le  2N-3 <  2^{N-1}-1$.  Thus ${\rm{SA}}_{n-1}\to O(\delta,\R)$
is trivial, as claimed.
This means that  ${\rm{SA}}_{n-1}$, which contains Nielsen transformations,  acts trivially on the tangent space at $p$.
It follows that ${\rm{SA}}_{n-1}$ fixes  every geodesic issuing from $p$,  and since every point of $X$ is 
joined to $p$ by a (unique) geodesic,  the action of ${\rm{SA}}_{n-1}$ on $X$ is trivial.
  But the kernel of the (non-trivial) action of ${\rm{SAut}}(F_{n})$ on $X$
 cannot contain a Nielsen transformation  because the normal
 closure of any Nielsen transformation is the whole group.  This contradiction completes the proof when $\delta\ge 2$.

It remains to rule out the possibility $\delta=1$. The only non-trivial isometry of a Hadamard manifold that
fixes a codimension-1 convex submanifold is orthogonal reflection in that submanifold. Thus
the restriction of the action ${\rm{SAut}}(F_n)\to \isom(X)$ to $M$ has image that is cyclic of order $2$.
In particular, this means that at least one of $\l_{n1},\, \l_{n2}, \, \l_{n1}\l_{n2}$ has trivial
image.   And since  $\l_{21}$ conjugates  $\l_{n2}$ to $\l_{n1}\l_{n2}$, this means that the kernel of the action 
contains a Nielsen transformation. 
As in the previous case, this contradicts the assumption  that  $\sautn$ is acting non-trivially.
\end{proof}

\begin{remark} From the standard representation ${\rm{SAut}}(F_{n})\to {\rm{SL}}(n,\Z)$,
one obtains an action with unbounded orbits of ${\rm{SAut}}(F_{n})$ by isometries on the Hadamard
manifold ${\rm{SL}}(n,\R)/{\rm{SO}}(n,\R)$, which has dimension $\frac 1 2 n(n+1)$, but this
action has parabolic isometries. If $n\le 3$, then there also exist semisimple actions
in certain dimensions \cite{mb:bill}, but I do not know of any non-trivial semisimple actions without a 
global fixed point for $n\ge 4$, and it seems possible that they might not exist  (cf.~Proposition \ref{p:sln}).
\end{remark}

\section{Forcing Nielsen Transformations to be Elliptic}

Our main result concerning fixed points for actions of $\aut$ is the following.

\firstthmA*

\begin{proof} By applying Proposition \ref{p:prodLemma} to 
the product of dihedral groups from Lemma \ref{l:dihedral}, we see that whenever
$\aut$ acts  in the given range, a Nielsen transformation will have a fixed point. 
It then follows from Theorem \ref{t:n-1} that $\sautn$ has a fixed point, and from
Lemma \ref{l:fi} that $\aut$ has a fixed point. 
\end{proof}

In the case $n\ge 3m$, the following result provides an alternative way of seeing that 
Nielsen transformations must have fixed points.  We have a fixed basis $\{x_1,\dots,x_n\}$ for $F_n$,
and by a {\em standard copy} of ${\rm{Aut}}(F_3)$ in \autn we mean
the group of automorphisms that leave a rank-3 free factor $\<x_i, x_j, x_k\>$ invariant and fix  the
remaining basis elements.

\begin{theorem}\label{t:elliptic} If $n\ge 3m$ and $d< 2m$, then whenever
\autn acts by isometries on a complete \cat space $X$
of dimension $d$, each standard copy of ${\rm{Aut}}(F_3)$ has a fixed point.
\end{theorem}

\begin{proof} If $n\ge 3m$, then \autn contains a direct product $D$ of $m$ standard
copies $\G_i$ of ${\rm{Aut}}(F_3)$, all of which are conjugate. Let $A_{i,1}, A_{i,2}, A_{i,3} < \G_i$
be the subgroups corresponding to the groups  $A_1,A_2,A_3\subset {\rm{Aut}}(F_3)$
described in Proposition \ref{p:A123}.  By construction, $\< A_{i,j}, A_{i,k}\>$ is finite for all $i,j,k$,
and $A_{i,j}$ commutes with $A_{k,l}$ when $i\neq k$. Thus we are in situation of 
Corollary \ref{p:triple} and 
deduce that one of the $\G_i$ has a fixed point in $X$. Each standard copy of ${\rm{Aut}}(F_3)$ in \autn
is conjugate to each $\G_i$, and hence has a fixed point.
\end{proof}

\section{Estimating the Fixed-Point Dimension of ${\rm{SL}}(n,\Z)$}

There is a well established and fruitful analogy between mapping class groups, automorphism
groups of free groups, and arithmetic lattices in semisimple Lie groups, particularly ${\rm{SL}}(n,\Z)$.
In previous sections we established constraints on the way in which the first two classes
of groups can act on \cat spaces, in this section we turn our attention to ${\rm{SL}}(n,\Z)$,
where the discussion is much more straightforward. Much of this straightforwardness can be traced
to the fact that all of the infinite cyclic subgroups of $\aut$ and $\modg$ are quasigeodesics in the
word metric on the ambient group, whereas the cyclic subgroups of $\sln$ generated by elementary
matrices are not (cf.~Proposition \ref{p:sln}).

We remind the reader that an {\em{elementary matrix}}
in \slnz $\,$ is a matrix of the form $E_{ij}=I_n +U_{ij}$, where $I_n$ is
the identity matrix and $U_{ij}$ is the matrix whose only
non-zero entry is a $1$ in the $(i,j)$-place with $i\neq j$. There is only one
conjugacy class of elementary matrices. It is not
difficult to show that the elementary matrices generate \slnz $\,$
but it is considerably more difficult to see that if $n\ge 3$ then they
{\em{boundedly generate}}: there exist elementary matrices
$E_1,\dots,E_N$ such that every $\gamma\in\sln$ can be expressed
as a product $\gamma = E_1^{p_1}E_2^{p_2}\cdots E_N^{p_N}$
for some $p_i\in\mathbb Z$; see \cite{bddGen}.

\begin{lemma}\label{Eenough}
 Let $n\ge 3$. If \slnz $\,$ acts by isometries on
a complete \cat space $X$ and some elementary matrix 
has a fixed point, then \slnz $\,$ has a fixed point.
\end{lemma}

\begin{proof} Since all elementary matrices are conjugate,
the hypothesis implies that each of the elementary matrices $E_i$
in the description of the bounded generation property has
a fixed point.  So in the light of Proposition \ref{bddOrbit},  we will be done if we can prove that
for any set of elliptic isometries $\mathcal{E}=\{e_1,\dots,e_N\}$ there is a function $C_{\mathcal{E}}:X\to\R$ such that
$d(x,\gamma(x))<C_{\mathcal{E}}(x)$ for all $\gamma\in\{e_1^{p_1}\cdots e_N^{p_N} \mid p_i\in\Z\}$ and  $x\in X$.
We argue by induction on $N$.
The case $N=1$ is trivial.  For the inductive step, we fix $x_0\in\fix(e_N)$ and note that
if $\gamma = e_1^{p_1}\cdots e_N^{p_N}$ then $\gamma(x_0)=\gamma'(x_0)$, 
where $\gamma'=e_1^{p_1}\cdots e_{N-1}^{p_{N-1}}$.  Let ${\mathcal{E'}}=\{e_1,\dots,e_{N-1}\}$. By induction,   
$d(x_0,\gamma(x_0))<{C_{\mathcal{E'}}}(x_0)$.  For an arbitrary $x\in X$,  by the triangle inequality, 
$$d(x, \gamma(x)) \le d(x, x_0) + d(x_0, \gamma(x_0)) + d(\gamma(x_0), \gamma(x)) = 2\, d(x,x_0) + d(x_0,\gamma(x_0)).$$
Defining $C_{\mathcal{E}}(x) := 2\, d(x,x_0) + {C_{\mathcal{E'}}}(x_0)$ completes the proof.
\end{proof}

In the years that have elapsed since the first draft of this article, the following observation
has been made independently by several authors.

\begin{prop}\label{p:sln} If $n\ge 3$, then $\sln$ fixes a point whenever it acts by semisimple
isometries on a complete \cat space.
\end{prop}

\begin{proof} If $n\ge 3$, the cyclic subgroup generated
by an elementary matrix $E_{ij}\in \sln$ is metrically distorted, in other words $\lim_{k\to\infty}\frac{1}{k}
d (1,E_{ij}^k) = 0$, where $d$ is the word metric associated to a finite generating set of $\sln$.
On the other hand, if a finitely generated group $\G$ acts by isometries on a complete \cat space, then
the cyclic subgroup generated by each hyperbolic isometry is undistorted in $\G$. 
Thus, whenever $\sln$ acts by semisimple isometries on a complete \cat space, the elementary matrices
have fixed points (i.e.~are elliptic isometries). Lemma \ref{Eenough} completes the proof.
\end{proof}

Farb \cite{farb} defines a group to be of type ${\rm{FA}}_d$ if it has a fixed point whenever it acts
by simplicial isometries on a piecewise Euclidean complex with finitely many isometry types of
cells that  is \cat and has dimension at most $d$. Simplicial isometries of such spaces are semisimple \cite{mb:pams}.

\begin{cor} If $n\ge 3$, then $\sln$ has property ${\rm{FA}}_d$ for every $d\in\N$.
\end{cor}

\subsection{Actions that are not semisimple} 

$\sln$ acts properly by isometries on the symmetric space ${\rm{SL}}(n,\R)/{\rm{SO}}(n,\R)$, which is 
a complete \cat manifold of dimension $\frac 1 2 n(n+1)$. Thus, in contrast to Proposition \ref{p:sln}, there
are interesting actions of $\sln$ on complete \cat spaces if one allows parabolics. In the
context of the present article, it is natural to ask in what dimensions $\sln$ has fixed-point free actions.
The best that I can offer is a lower bound that is linear in $n$.
 
\begin{proposition}\label{t:sln} $\ $
\begin{itemize}
\item $\fd({\rm{GL}}(n,\Z)) \ge n-2$.
\item $\fd({\rm{SL}}(n,\Z)) \ge n-2$ if $n$ is odd.
\item $\fd({\rm{SL}}(n,\Z)) \ge n-3$ if $n$ is even.
\end{itemize}
\end{proposition}

\begin{proof} Let $\tau_i$ be the diagonal matrix
with $-1$ in the $i$-th place and ones elsewhere. Then
$\tau_i$ commutes with $E_{jn}$ for $j\neq i$ and conjugates
$E_{in}$ to $E_{in}^{-1}$.  Thus ${\rm{GL}}(n,\Z)$ contains a direct
product of $(n-1)$ copies of the infinite dihedral group, namely
$D_{(i)} = \langle \tau_i,\, E_{in}\rangle$ with $i=1,\dots,n-1$.

Consider an action of ${\rm{GL}}(n,\Z)$ by  isometries on a complete
\cat space $X$ of dimension $\le n-2$. 
According to Proposition \ref{p:prodLemma}, one of the $D_{(i)}$
must have a fixed point in $X$, so by Lemma \ref{Eenough}, ${\rm{SL}}(n,\Z)$ has a fixed point,
and by Lemma \ref{l:fi},   ${\rm{GL}}(n,\Z)$ does too.

If $n$ is odd, we can repeat this argument with $\<-I_n\tau_i, E_{in}\>\le {\rm{SL}}(n,\Z)$  in place of $D_{(i)}$
to see that $\fd({\rm{SL}}(n,\Z)) \ge n-2$. If $n$ is even, then we drop $D_{(1)}$ and replace
$D_{(i)}$ by $\<\tau_1\tau_i, E_{in}\>\le {\rm{SL}}(n,\Z)$  for $i=2,\dots,n$.
\end{proof}
 
 Similar arguments apply to  other lattices in higher-rank groups, cf.~\cite{farb}

\section{Appendix on Theorem \ref{t:helly}}

In this appendix, we prove a Helly-type result that contains Theorem \ref{t:helly} as a special case.
There are
many similar theorems in the literature and it is unlikely that anything here will be unfamiliar to
experts. But one has to be careful about the exact hypotheses 
if one wants to apply such theorems to spaces that potentially have  wild local structure, as we do
in this article; this is discussed in  \cite{BKV}.

Let $X$ be a metric space and let $\C$ be a 
finite collection $C_0,\dots,C_n$ of
closed, non-empty subsets of $X$. We shall say that  a non-empty indexing subset
$J\subseteq \n = \{0,\dots,n\}$ is  {\em{admissible}}    if
$\cap_{j\in J} C_{j}\neq\emptyset$. 
We write 
$\NC$ to denote the {\em{nerve}}
of $\C$, i.e.~the simplicial complex
with vertex set $\n$ that has an
$r$-simplex $\sigma_J$ with vertex set $J$ for each 
admissible $(r+1)$-element
subset $J\subseteq\n$. (It is sometimes convenient to regard $\NC$ as a 
subcomplex of the standard $n$-simplex $\Delta_n$.)

The set of admissible
subsets, which we denote by $\Sigma(\C)$, is partially ordered by inclusion
and the geometric realization of this poset is the first
barycentric subdivision of $\NC$, denoted $\NC'$.  

For admissible subsets $J\subseteq \n$ we write $C^J=\bigcap_{j\in J}
C_j$ and $C_J=\bigcup_{j\in J} C_j$.

We say that $\C$ is 
{\em{sufficiently connected}}\footnote{One can manufacture less restrictive but more technical definitions that
suffice for Proposition \ref{inverse}; this choice is a compromise
between technicality and utility.} if for each admissible set $I\subseteq\n$ the intersection $C^I$ is connected if $I$
is maximal and $h(I)$-connected otherwise,   where  
$h(I) = \max\{ |J|-|I| \colon I\subset J\in\Sigma(\C)\} -1$.   

\begin{lemma}\label{l:basic} If $\C$ is sufficiently connected, then 
\begin{enumerate}
\item every sub-collection $\C'\subset \C$ is sufficiently connected, and
\item for every $C_0\in\C$, the collection $\{C\cap C_0\mid C\in\C\}\ssm\{\emptyset\}$ is sufficiently connected.
\end{enumerate}
\end{lemma}

\begin{prop}\label{inverse} 
If $\C$ is
sufficiently connected then there exists a compact set $U\subseteq
X$ and continuous maps $\phi: \NC\to U$ and $\psi: U\to \NC$
such that $\psi\circ\phi:\NC\to \NC$ is homotopic to the
identity.
\end{prop}
  
We first construct $\phi:\NC\to X$.

\begin{lemma}\label{phi} If $\C$ is sufficiently
connected then there exists
a continuous map $\phi: \NC\to X$ such that $\phi(\sigma_J)
\subseteq C_J$
for all $J\in\Sigma(\C)$.
\end{lemma}

\begin{proof}  A typical $m$-simplex $S$ of  the barycentric subdivision 
$\NC'$ has vertices
$b(J_0),\dots,b(J_m)$, where $J_0\subset\cdots\subset J_m$
and $b(J_i)$ denotes the barycentre
of $\sigma_{J_i}$. 
We  will construct $\phi$ inductively on
the skeleta of $\NC'$, ensuring that $\phi(S)$ is contained
in   $C^{J_0}$.   (If $J_i\subseteq J$ then $C^{J_i}\subseteq C_J$, so
$\phi(S)\subseteq C^{J_0}$ implies that $\phi(\sigma_J)\subseteq C_J$, as required.)

In the base step of the induction we can choose $\phi(b_{J_i})\in
C^{J_i}$ arbitrarily. Assume, then, that $\phi$ has been
defined on $\partial S$  so that for each face $V<S$ we have $\phi(V)\subseteq C^{J_v}$
where $J_v$ is the smallest vertex of $V$.  Then, $J_0\subseteq J_v$ implies 
 $\phi(\partial S)\subseteq C^{J_0}$.  As
$\partial S$ is a topological sphere of
dimension at most $|J_m|-|J_0|\le h(J_0)$ and
$C^{J_0}$ is assumed to be $h(J_0)$-connected,
$\phi$ can be extended to a continuous map
on $S$ with image in $C^{J_0}$ and the induction is complete.
\end{proof}

Let $\phi$ be as above and let $U\subseteq X$ be the image of 
$\phi$.
We wish to construct a map $\psi:U\to \NC$ such that
$\psi\circ\phi\simeq {\hbox{id}}_{\NC}$.
If the intersection of the entire collection $\C$
is non-empty, then $\NC$ is an $n$-simplex and
we can define $\psi$ to be any choice of constant map.
Thus we may assume that the $C_i$ do not have a point
of common intersection. It follows that for each
$x\in X$ the set $J(x):=\{j\mid x\in C_j\}$ is a
proper subset of $\n$.
Let $\e(x)=\min_{i\notin J(x)}
d(x,C_i)$. Let $B(x)\subseteq X$ be the open ball
of radius $\e(x)$ about $x$ and
note that if $y\in B(x)$ then $J(y)\subseteq J(x)$.

By the compactness of $U= \im(\phi)$,
there is a finite set $T$ such that the union 
of the balls $B(x_t),\, t\in T,$ 
contains $U$,  with $x_t\in U$. We consider the nerve $\NT$
of the collection $\mathcal T$
of sets $\{B(x_t)\cap U\colon t\in T\}$.  As before, the barycentric subdivision $\NT'$ is 
naturally the
geometric realization of the poset of admissible subsets
of $T$ ordered by inclusion, which we denote $\Sigma(\mathcal T)$.
However, it is now more convenient to reverse the face relation
and regard the space $\NT'$ as the geometric realization of the poset
$\Sigma^{op}(\mathcal T)$. 

For $\tau\in \Sigma(T)$ we define $J[\tau] := \bigcap_{t\in\tau} J(x_t)$.

\begin{lemma} 
$\tau\mapsto J[\tau]$ defines a morphism of 
posets
$$
\Sigma^{op}(\T)\to\Sigma(\C)
$$
and hence induces a continuous map on their geometric realizations
$$
\Psi: \NT'\to\NC'.
$$
\end{lemma}

\begin{proof}
If $y\in U$ lies in the intersection of the balls $B(x_t),\, t\in\tau$, then $J(y)$ is a non-empty
subset of $J(x_t)$ for all $t\in\tau$.
Thus
$J[\tau]$ is a non-empty
subset of an admissible set, hence is admissible.
It is clear that if $\tau_1\supseteq\tau_2$ then
$J[\tau_1]\subseteq J[\tau_2]$. 
\end{proof}

The next part of the argument is modelled on Lemma I.7A.15 in \cite{BH}.
The map $\psi: U\to \NC$ that we seek is obtained by composing
$\Psi$ with a map $f:U\to \NT$ constructed using
a partition of unity subordinate to the covering $\T$.
Thus for each $t\in T$,
we define $f_t:  U\to [0,\infty)$ by 
$f_t(y)=\e(x_t) - d(y,x_t)$
if $y\in B(x_t)$ and $f(y) = 0$
otherwise. We then define a continuous map 
$f: U\to \NT$ by sending $y$ to the point whose
$t$-th barycentric coordinate is 
$$
\frac{f_t(y)}{\sum_{s\in T}f_s(y)}.
$$

The following lemma completes the proof of Proposition \ref{inverse}.

\begin{lemma} $\Psi\circ f\circ \phi$ is homotopic to
the identity of $\NC$.
\end{lemma}

\begin{proof} It is enough to prove that each simplex $\sigma_J$
of $\NC$ is mapped into the union of the open stars of its vertices $j\in J$.
One can then construct a homotopy to the identity by proceeding
one simplex at a time using the obvious ``straight-line homotopies",   see \cite{spanier} (3.3.11).

If $p\in\sigma_J$ then by construction $\phi(p)\in C_j$ for some $j\in J$, 
so $j\in J(\phi(p))$.  
The vertex set $\tau$ of the open simplex in $\NT$ containing $f(\phi(p))$ 
consists of those vertices $t$ for which  $f_t(\phi(p))>0$;
equivalently
$\phi(p)\in
B(x_t)$, which implies $J(\phi(p))\subseteq J(x_t)$. 
Thus $j\in J[\tau]$,  which means that  $\Psi\circ f\circ\phi(p)$ lies
in a simplex of $\NC'$ that  is in the  closed star neighbourhood of $\{j\}$,
which is in the open star of $j$ in $\NC$.
\end{proof}

We recall the definition of dimension that is most convenient to our ends.

\begin{definition} A topological space $X$ has dimension
at most $d$, written $\dim(X)\le d$, if for every closed
subspace $K\subseteq X$, every continuous map
$f:K\to\S^r$ to a sphere of dimension $r\ge d$ extends to
a continuous map $X\to\S^r$.
\end{definition}

\begin{definition} A topological space $X$ is {\em $d$-coconnected} if every
continuous map from $X$ to a sphere of dimension $r\ge d$ is homotopic to a constant map.
\end{definition}
 
 For example, the $n$-sphere is $(n+1)$-coconnected but not $d$-connected for $d\le n$.
 Contractible spaces are $d$-coconnected for all $d$.
 Theorem \ref{t:helly} is a special case of the following version of Helly's theorem.

\begin{theorem} \label{nerve} Let $d\ge 0$ be an integer.
If $X$ is a $d$-coconnected metric space of 
topological dimension $\le d$ and  $\C$ is a sufficiently
connected collection of closed subsets of $X$, with
nerve $\NC$, then every continuous map $\NC\to\S^r$ to a
sphere of dimension $r\ge d$ is homotopic to a constant map.
\end{theorem}

\begin{proof} Suppose $r\ge d$ and consider a 
continuous map $g:\NC\to\S^r$. Proposition \ref{inverse} provides
a compact  $U\subseteq X$ and continuous maps $\phi:\NC\to U$
and $\psi:U\to\NC$ such that $\psi\circ\phi\simeq {\hbox{id}}_{\NC}$.
Since $X$ is $d$-dimensional, $g\circ\psi$ has a continuous
extension to $X$, and since $X$ is $d$-coconnected this extension
is homotopic to a constant map. Thus $g\circ\psi$ and $g\circ\psi\circ\phi
\simeq g$ are homotopic to constant maps.
\end{proof}

\end{document}